\documentclass{amsart}
\usepackage{amssymb,amsmath}
\usepackage{amsfonts}

\newtheorem{thm}{Theorem}[section]
 \newtheorem{cor}[thm]{Corollary}
 \newtheorem{lem}[thm]{Lemma}
 
 \theoremstyle{definition}
 \newtheorem{defn}[thm]{Definition}
 \theoremstyle{remark}
 \newtheorem{rem}[thm]{Remark}
 
 \newtheorem*{hp}{Hypothesis}
 \newtheorem{ex}[thm]{Example}
 \numberwithin{equation}{section}

\begin{document}

\title[On the connectivity of the Sylow graph of a finite group]{On the connectivity of the Sylow graph of a finite group}
\author[F.G Russo]{Francesco G. Russo}
\address{
University of Palermo\newline 
\indent Department of Mathematics \newline
\indent via Archirafi 34, 90123, Palermo, Italy}
\email{francescog.russo@yahoo.com}



\date{\today}


\maketitle



\begin{abstract}
The Sylow graph $\Gamma(G)$ of a finite group $G$ originated from recent investigations on the so--called $\mathbf{N}$--closed classes of groups. The connectivity of $\Gamma(G)$ was proved only few years ago, involving the classification of finite simple groups, and the structure of $G$ may be strongly restricted, once information on $\Gamma(G)$ are given.  The first result of the present paper deals with a condition on $\mathbf{N}$--closed classes of groups. The second  result deals with a computational criterion, related to the connectivity of $\Gamma(G)$.
\end{abstract}

\medskip


\section{Sylow graph and formations of groups}\label{s:1}

All the groups, which are considered in the present paper, are
finite. If $G$ is a group in which every Sylow subgroup is
self--normalizing, then \cite[Theorem 1]{gg2} shows that $G$ is a
$p$--group for some prime $p$. \cite[Theorem 1]{gg2} is a classical
result and belongs to a long standing line of research which
investigates the structural properties of a group, once restrictions
on its normalizers are given (see \cite[Chapter 5]{dh}). As noted in
\cite{bgh}, we may introduce the group class operator
$\mathbf{N}$:
\begin{equation}\mathbf{N}(\mathfrak{X})=(G: N_G(G_p)\in \mathfrak{X}, \forall
p\in \pi(G)),\end{equation} where $\mathfrak{X}$ is a class of groups, $\pi(G)$ is the set of prime divisors of
the order of $G$ and $G_p$ denotes the Sylow $p$-- subgroup of $G$, briefly $G_p\in \mathrm{Syl}_p(G)$.
$\mathbf{N}$ was largely studied in \cite{b1, b3, bgh, d1, d2, d3, d4, d5, kmp}. Originally,
\cite[Corollary 2]{gg2} shows that the class $\mathfrak{E}_p$ of all
$p$--groups is $\mathbf{N}$--closed. More generally, \cite[Theorem 2]{bgh} shows that the class $\mathfrak{N}$ of all nilpotent groups is $\mathbf{N}$--closed. A further improvement is \cite[Theorem 2]{b1},
where it is considered the class of all $p$--nilpotent groups. In the class $\mathfrak{S}$ of all solvable
groups,  relations among $\mathbf{N}$ and classes of groups which are closed with respect to forming subgroups (briefly, $\mathbf{S}$--closed), were investigated in \cite{d1,d2}. These researches continued in
\cite{d3, d4, d5}, introducing some technical notions, which refer to ideas and techniques in the theory of
formations of groups due to Gasch\"utz, Lubesender and Shemetkov. The terminology is standard in literature and can be found in \cite{dh}. A significant notion
is the following.

\begin{defn}[See \cite{d2}]\label{d:1} Assume that $\pi$ is a set of primes.
A covering $\mathcal{R}=(\pi(p))_{p \in \mathbb{P}}$ is called
$symmetric$ in $\pi$, if the following properties hold:
\begin{itemize}
\item[(i)]  $\pi={\underset{p\in \pi}\bigcup}\pi(p),$ \item[(ii)]
$p\in \pi(p)$, for each $p\in \pi$, \item[(iii)] if $q\in \pi(p)$,
then $p\in\pi(q)$.
\end{itemize}
Let $f_{\mathcal{R}}$ be the formation function, defined by
\begin{equation}f_{\mathcal{R}}(p) = \left\{\begin{array}{lcl} \mathfrak{E}_{\pi(p)}, \,\,\mathrm{if}\,\,p\in \pi\\
\emptyset, \,\,\,\,\,\,\,\,\,\,\mathrm{if}\,\,p\not\in
\pi.\end{array}\right.\end{equation} The saturated formation
$LF(f_{\mathcal{R}})=\mathfrak{E}_{\mathcal{R}},$ locally defined by
$f_{\mathcal{R}}$, is said to be a $covering$--$formation$,
associated to $\mathcal{R}$. The intersection
$\mathfrak{E}_{\mathcal{R}}\cap \mathfrak{S}$ is said to be a
covering--formation of solvable groups, associated to $\mathcal{R}$.
If $\mathcal{R}$ is a partition of $\pi$,
$\mathfrak{E}_{\mathcal{R}}$ is called $lattice$--$formation$.
\end{defn}

In \cite{b1, b3, bgh, d1, d2, d3, d4, d5} the
covering--formations and the lattice--formations originated the
following graph, which appears in \cite{d5} for the first time.

\begin{defn}[See \cite{d5}]\label{d:2}
Let $G$ be a group and $p, q$ two distinct primes in $\pi(G)$. Then
there is a sequence $p_1=p, p_2, \ldots, p_n=q$ in $\pi(G)$ such
that either $p_i$ divides $|N_G(G_{p_{i+1}}):C_G(G_{p_{i+1}})|$ or
$p_{i+1}$ divides $|N_G(G_{p_i}):C_G(G_{p_i})|,$ for each $i=1,2,
\ldots,n$. When this happen, we write briefly $p\approx q$.
$\Delta(G)$ denotes the graph with set of vertices $\pi(G)$ and
edges given by the relation $\approx$.
\end{defn}
The connectivity of $\Delta(G)$ was hard to investigate in the form
of Definition \ref{d:2}. In fact it was reformulated under a
different prospective by Kazarin and others in \cite{kmp}. These
authors study the following graph, known as the $Sylow$ $graph$ $of$
$G$.

\begin{defn}[See \cite{kmp}] \label{d:3}
In a group $G$ define the automiser $A_p(G)$ to be the group
$N_G(G_p)/G_pC_G(G_p)$. The Sylow graph $\Gamma(G)$ of $G$, with set
of vertices $\pi(G)$, is given by the following rules: Two vertices
$p, q \in \pi(G)$ form an edge of $\Gamma_A(G)$ if either $q \in
\pi(A_p(G))$ or $p \in \pi(A_q(G))$. Briefly, we write $p
\rightarrow q$ if $q \in \pi(A_p(G))$.
\end{defn}
It is easy to see that $p\rightarrow q$ implies $p\approx
q$, since $|N_G(G_p):G_pN_G(G_p)|\leq |N_G(G_p):C_G(G_p)|$. In
particular, if $G_p$ is abelian, then $G_p\subseteq C_G(G_p)$ and we
conclude that $\Delta(G)=\Gamma(G)$. Therefore the two graphs are
related and we can summarize all in the following observation. 

\begin{rem}While $\Delta(G)$ admits loops, $\Gamma(G)$ does not admit loops. This fact does not affect the connectivity of $\Delta(G)$ and $\Gamma(G)$ so that $\Delta(G)$ is connected if, and only if, $\Gamma(G)$ is connected.
\end{rem}

The connectivity of $\Gamma(G)$ was  proved in \cite[Main Theorem]{kmp} and influenced heavily the first versions of the present paper. An immediate consequence, conjectured in \cite{d5}, is that the lattice--formations are $\mathbf{N}$--closed. Therefore $\Gamma(G)$ allows us to decide whether a class of groups is $\mathbf{N}$--closed or not.
In Section 2 we describe some abstract conditions, which are useful for the same scope, without using the methods of the graph theory. Since
the properties of $\Gamma(G)$ are related to those of the Sylow
normalizers of $G$, it is used the classification
of finite simple groups in \cite{kmp}. In Section 3 we illustrate a  criterion for the connectivity of $\Gamma(G)$, dealing with a computational method. In this way, we avoid the classification of finite simple groups. 

\section{Some abstract conditions}
We begin with some examples to become confident with
$\mathbf{N}$--closed classes.

\begin{ex}[See \cite{kmp}, Examples 2,3 and Remark in Section 5]\label{e:1} Consider $\pi=\{2,3,5\},$ $\pi(2)=\{2,3,5\},$
$\pi(3)=\{2,3\},$ $\pi(5)=\{2,5\}$ and
$\mathcal{R}=\pi(2)\cup\pi(3)\cup \pi(5)$. Definition \ref{d:1} is
satisfied and $\mathcal{R}$ is symmetric in $\pi$. Note that the
alternating groups on 5 elements belong to
$\mathbf{N}(\mathfrak{E}_{\mathcal{R}}) \setminus
\mathfrak{E}_{\mathcal{R}}$. This suggests that there is not a
characterization of covering--formations in the universe
$\mathfrak{E}$ of all finite groups in terms of $\mathbf{N}$.
\end{ex}
Example \ref{e:1} shows the following fact. If $E$ is a simple
group, belonging to a covering--formation, then $E$ is a
$\pi(p)$--group, for each $p\in \pi(E)$, because $E$ is clearly
the only chief factor of itself.

\begin{rem}\label{r:1} Still from Definition \ref{d:1}, the lattice--formation, related to the minimal partition of $\pi$,
is the formation $\mathfrak{N}_{\pi}$ of all nilpotent
$\pi$--groups. In particular,
$\mathfrak{N}_{\mathbb{P}}=\mathfrak{N}$. The lattice--formation,
related to the maximal partition of $\pi$, is $\mathfrak{E}_\pi$. In
particular, $\mathfrak{E}_{\mathbb{P}}=\mathfrak{E}$. From
\cite{b3}, the lattice--formation, related to a partition
$(\pi_i)_{i\in I}$ of $\pi\subseteq \mathbb{P}$, is the class of the
groups which are direct product of $\pi_i$--groups.
\end{rem}

\begin{ex}\label{e:2}
Still from Definition \ref{d:1}, the class $\mathfrak{D}$ of all
groups which have odd order is both a lattice--formation and a
covering--formation of solvable groups, thanks to a well--known
result of J. Thompson and W. Feit. More precisely,
$\mathfrak{D}=LF(f_{\mathcal{R}})$, where
$f_{\mathcal{R}}(2)=\emptyset$ and
$f_{\mathcal{R}}(p)=\mathfrak{S}_{2'}$ if $p\not=2.$ Note that
$\mathfrak{D}=\mathfrak{S}_{2'}$ and the characteristic of the
formation $\mathfrak{D}$ (see \cite{dh} for the terminology)  is
equal to $2'=\mathbb{P}\setminus \{2\}$.
\end{ex}
The above considerations show that saturated formations, which are
$\mathbf{S}$--closed and $\mathbf{N}$--closed, could not be
covering--formations. This was already noted in \cite[Section
3]{d4}. The following condition is more easy to check in a computational way.

\begin{hp}\label{h:1} Assume that in a group $G$, 
if  $p\in \pi(G)$ and $p$ is odd, then there exists a prime $q$ such that $q<p$ and $q$ divides
$|N_G(G_p):C_G(G_p)|$ for each $G_p\in \mathrm{Syl}_p(G)$.
\end{hp}

If Hypothesis is true, then $\Delta(G)$ is connected and, consequently, $\Gamma(G)$ is connected. Unfortunately, the converse is false, as shown below.

\begin{ex}\label{e:3} Let $H=PSL_2(3^{3^k})$ for any integer $k\geq1$ and
consider $G=H\langle \sigma \rangle$, where
\begin{equation}
\sigma : \left(
\begin{array}{ccccccc}
a & b\\
c & d\\
\end{array} \right)
\in PSL_2(3^{3^k})\longmapsto \left(
\begin{array}{ccccccc}
a^3 & b^3\\
c^3 & d^3\\
\end{array}\right)\in PSL_2(3^{3^k}).
\end{equation}
$G$ is an almost--simple group, i.e.: $\mathrm{Inn} (S) \leq G \leq \mathrm{Aut} (S)$ for some non--abelian
simple group $S$, and $\{2,3\}\subseteq \pi(G)$. Now $G$ does not satisfies Hypothesis, because  $2$ does not
divides $|N_G(G_3):C_G(G_3)|$. At the same time, one can see that there is a prime $l\in \pi(G) \setminus
\{2,3\}$ such that $2$ divides $|N_G(G_l):C_G(G_l)|$ and $3$ divides $|N_G(G_l):C_G(G_l)|$ so that $2\approx 3$.
\end{ex}

Saturated formations, which are both $\mathbf{S}$--closed and
$\mathbf{N}$--closed, are characterized in \cite[Theorem 2]{d4} and
yields to interesting constructions as in \cite[Section 3]{d4} or in
\cite[Section 5]{kmp}. We recall just one case.

\begin{ex}\label{e:fundamental} Consider a set of primes $\pi=\{p\}\cup(\pi\setminus\{p\})=\{p\} \cup \omega$ such that there exist a non--abelian simple
$\omega$--group. Then the formation function $f$, defined by $f(q) =
\emptyset$, if $q\not\in\pi$;  $f(q)=\mathfrak{S}_{\pi}$, if $q=p$;
$f(q)= \mathfrak{E}_{\pi}$, if $q\in\omega$, determines the
saturated formation $\mathfrak{F}=LF(f)$, locally defined by $f$,
which is in many situations both $\mathbf{S}$--closed and
$\mathbf{N}$--closed but not a covering--formation. See
\cite[Section 3]{d4} for details.
\end{ex}

We may extend Example \ref{e:fundamental} in a more appropriate way. Recall
that $G^\infty$ denotes the largest perfect subgroup of a group $G$.
Recall also that $\mathfrak{F}$ in Example \ref{e:fundamental} is the class of
all finite groups $G$ such that whenever $H/K$ is a chief--factor of
$G$ and $q$ is a prime dividing $|H/K|$, then $G/C_G(H/K)\in f(q)$.

\begin{lem} \label{l:1}  $G\in \mathfrak{F}$ if and only if
$G$ is a $\pi$--group and $G^\infty$ is a $p'$--group.
\end{lem}
\begin{proof}
Suppose $G\in \mathfrak{F}$. If $q$ is a prime divisor of $|G|$ then
there exists a chief factor $H/K$ of $G$ such that $q$ divides $
|H/K|$. Then $G/C_G(H/K)\in f(q)$. Thus $f(q)$ is not empty and so
$q\in \pi$. So $G$ is a $\pi$--group.
Suppose for a contradiction that $G^{\infty}$ is not a $p'$--group.
Then there exists a chief factor $H/K$ of $G$ such that $H\leq
G^\infty$, $p$ divides $ |H/K|$ and $G^{\infty}/H$ is a $p'$--group.
It follows that $G/C_G(H/K)\in f(p)$ and so $G/C_G(H/K)$ is
solvable. Thus $G^{\infty}\leq C_G(H/K)$ and so $H/K\leq
Z(G^{\infty}/K)$. Since $p$ divides $|H/K|$, $H/K$ is a
$p$--elementary abelian group. Since $G^{\infty}/K$ is $p'$--group,
we conclude that $G^{\infty}/K=H/K\times L/K$  for some subgroup
$L/K$ of $G^{\infty}/K$. This is a contradiction, since
$G^{\infty}/K$ is perfect and $H/K$ is abelian.

Suppose next that $G$ is a $\pi$--group and $G^{\infty}$ is a
$p'$--group. Let $H/K$ be a chief factor of $G$ and $q$ a prime
dividing $|H/K|$. Since $G$ is a $\pi$--group, $q\in \pi$ and
$G/C_G(H/K)$ is a $\pi$--group.  So if $q\neq p$ we have
$G/C_G(H/K)\in f(q)$. Suppose $q=p$. Since $[H/K,G']=[H,G']K/K$ is a
$G$--invariant $p'$--subgroup of $H/K$, we conclude that $G'\leq
C_G(H/K)$ and so $G/C_G(H/K)$ is solvable. Hence again
$G/C_G(H/K)\in f(q)$. So $G\in \mathfrak{F}$.
\end{proof}

\begin{cor} $\mathfrak{F}$ is $\mathbf{S}$--closed. \end{cor}
\begin{proof}This follows immediately from Lemma \ref{l:1}.
\end{proof}
In the next results we will use the $generalized$ $Fitting$
$subgroup$ $F^*(G)$ of a group $G$ and the $components$ $of$ $G$.
These are well--known notions, which can be found in
\cite[p.580]{dh}.

\begin{lem}\label{l:2} Let $H/K$ be a non--abelian chief--factor
of a group $G$, $L$ a component of $H/K$,
$\overline{U}=N_{G/K}(L)/C_{G/K}(L)$, $q$ a prime, $T\in
\mathrm{Syl}_q( \overline{U})$  and $S\in \mathrm{Syl}_q(G)$. Then
$F^*(\overline{U})=\overline{L}\cong L$ and $N_{\overline{
L}}(T)/\overline{L}\cap T$ is isomorphic to a subgroup of
$N_H(S)/(S\cap H)K$.
\end{lem}

\begin{proof} The first statement is obvious. To simplify the notation we may
replace $G$ by $G/K$ and assume that $K=1$. Put $R=N_S(L)$. We also
may assume that $T=\overline{R}$.  Then
$N_{\overline{L}}(T)/\overline{L}\cap T\cong N_L(R)/L\cap R$. Let
$s_1,\ldots s_k$ be a transversal to $R$ in $S$ with $s_1=1$ and put
$L_i=L^{s_i}$. Then $M=\langle L^S\rangle=L_1\times L_2\times \ldots
L_k$ and it is now easy to see that $N_M(S)/S\cap M\cong
N_L(R)/L\cap R$.\end{proof}
\begin{cor}\label{c:2}Let $G$, $L$ and $\overline{U}$ be as in Lemma
\ref{l:2}. If $G \in \mathbf{N}(\mathfrak{F})$, then $\overline{U}
\in \mathbf{N}(\mathfrak{F})$.
\end{cor}
\begin{proof} Since $\mathfrak{F}$ is closed under sections, this follows
from Lemma \ref{l:2}.
\end{proof}
Note that a group $U$ is almost--simple if and only if $F^*(U)$ is
simple.
\begin{thm}\label{t:2}$\mathfrak{F}=\mathbf{N}(\mathfrak{F})$ if and only the following statements
are true:
\begin{itemize} \item[(a)] If $U$ is a non--solvable almost--simple group in $\mathbf{N}(\mathfrak{F})$,  then $F^*(U)$ is a
$p'$--group. \item[(b)] If $L$ is a non--abelian finite simple
$(\pi\setminus \{p\})$--group and $S\in
\mathrm{Syl}_p(\mathrm{Aut}(L))$, then $C_L(S)$ is non--solvable.
\end{itemize}
\end{thm}

\begin{proof} Suppose first that $\mathfrak{F}=\mathbf{N}(\mathfrak{F})$. Let  $U$ be an
almost--simple group in $\mathbf{N}(\mathfrak{F})$. Since
$\mathfrak{F}=\mathbf{N}(\mathfrak{F})$, we conclude that $L\in
\mathfrak{F}$ and so $U^\infty$ is a $p'$--group. Since $U$ is
non--solvable, $F^*(U)$ is perfect and so  $F^*(U)\leq U^\infty$ and
$F^*(U)$ is a $p'$--group.
Let $L$ be a finite simple $(\pi\setminus \{p\})$--group and $S\in
\mathrm{Syl}_p(\mathrm{Aut}(L))$. Let $\mathbb{F}$ be a finite field
of characteristic $p$ such that $q$ divides $|\mathbb{F}^\sharp|$
for all $q\in \pi$ with $q\neq p$ ($\mathbb{F}^\sharp$ denotes the
multiplicative group of $\mathbb{F}$). Let $D$ be the $(\pi\setminus
\{p\})$--Hall subgroup of $\mathbb{F}^\sharp$. Let  $V$ be a
faithful $\mathbb{F}LS$--module. Then $V$ is also a module for
$D\times LS$ and we can consider the semidirect product $H=V(D\times
LS)$. Note that $H$ is a $\pi$--group. Also $[V,L]\leq H^\infty$ and
so $H^\infty$ is not a $p'$--group. So $H\not\in \mathfrak{F}$.
Since $\mathfrak{F}=\mathbf{N}(\mathfrak{F})$, this gives $H\not\in
\mathbf{N}(\mathfrak{F})$. Hence there exists $q\in \pi$ and $T\in
\mathrm{Syl}_q(H)$ such that $N_H(T)\not\in \mathfrak{F}$. Since
$N_H(R)$ is a $\pi$--group, this means that $N_H(T)^\infty$ is not a
$p'$--group. Since $N_H(T)^\infty\leq H^\infty\leq VL$ and $L$ is a
$p'$--group we conclude that $N_H(T)\cap V\neq 1$.
 Assume that $q=p$. Then $T\cap
V=1$ and $1\neq N_H(T)\cap V=C_V(T)=C_V(VT)$. Since $T$ is a Sylow
$q$--subgroup of $H$ and $q$ divides $|D|$, we get $VT\cap D\neq 1$.
But each non--trivial element of $D$ acts fixed point freely on $V$
and so $C_V(VT\cap D)=1$, a contradiction to $C_V(VT)\neq 1$.
Thus $q=p$ and we may assume that $T=VS$. Then $N_H(T)=VDN_L(S)$
and, since $N_H(T)$ is non--solvable, we conclude that $N_L(S)$ is
non--solvable.
\medskip
Suppose next that  (a) and (b) hold.  We will first show that
\begin{equation}\label{claim:11}  \mathrm{Any} \ \mathrm{non}-\mathrm{abelian} \   \mathrm{composition} \  \mathrm{factor} \  L \ \mathrm{of} \  G \in
\mathbf{N}(\mathfrak{F}) \  \mathrm{is} \  \mathrm{a}  \
p'-\mathrm{group}.\end{equation}
 Indeed by Lemma \ref{l:2} and Corollary \ref{c:2}, $L\cong
F^*(\overline{U})$ for the simple  group $\overline{U} \in
\mathbf{N}(\mathfrak{F})$. Thus by (a), $L$ is a $p'$--group. So
(\ref{claim:11}) holds.

Suppose for a contradiction that there exists $G\in
\mathbf{N}(\mathfrak{F})\setminus \mathfrak{F}$. Then  there exist a
chief--factor $H/K$ and a prime $q$ dividing $|H/K|$ such that
$G/C_G(H/K)\not\in f(q)$. Since $G$ is a $\pi$--group, this implies
that $q=p$ and $G/C_G(H/K)$ is non--solvable. Thus $G^\infty\nleq
C_G(H/K)$ and there exists a normal subgroup $M$ of $G$ maximal with
$C_G(H/K)\leq M<G$.  Observe that $G^\infty/M$ is a non--abelian
chief factor of $G$.   Let $L$ be a composition factor of
$G^\infty/M$. By (\ref{claim:11}), $L$ is a $p'$--group.  From (b)
and Lemma \ref{l:2} we conclude that $N_{G^\infty}(S)M/M$ is
non--solvable and so
\begin{equation}\label{claim:12}N_G(S)^\infty\nleq
C_G(H/K).\end{equation}
 Since $G\in \mathbf{N}(\mathfrak{F})$,
$N_G(S)^\infty$ is a $p'$--group. Since $[N_G(S)^\infty,S]\leq
N_G(S)^\infty\cap S,$ we conclude that  $N_G(S^\infty)\leq C_G(S)$.
By (\ref{claim:11}), $H/K$ is a $p$--group. Thus $H=(S\cap H)K$ and
so \begin{equation}[N_G(S)^\infty,H]\leq
[N_G(S^\infty),S]K=K,\end{equation} which is in contradiction with
(\ref{claim:12}).
\end{proof}
\begin{rem} (a) and (b) in Theorem \ref{t:2} follow from:
\medskip
(c)\quad   If $L$ is a finite simple $\pi$--group, then
$\mathrm{Aut}(L)$ is a $p'$--group.
\end{rem}

\section{A computational method}
In the present section we illustrate an algorithm in GAP (see \cite{gap}), which allows us to decide  whether $\Delta(G)$ is connected or not in an empiric way. The times of answer of the oracle are quite long, but, in case we use Hypothesis, they are significatively reduced. The present algorithm has been called ROSN1, 
achronymus of "Restrictions On Sylow Normalizers". In Appendix the reader may find the description of $\Delta(G)$ for all the sporadic simple groups ($\textrm{Div}(k)$ denotes the set of divisors of the integer
$k\geq1$). These results are deduced by looking at \cite{atl1} and confirm the empiric results which we get on ROSN1.  \\
\\
\texttt{TestGroup:=function(G)\\
\\
local pi,p,S,x,N,C,Q,R,i,res;\\
\medskip
pi:=AsSet(FactorsInt(Size(G))); \\
S:=List(pi,p->SylowSubgroup(G,p));\\
N:=List(S,x->Normalizer(G,x));\\
C:=List(S,x->Centralizer(G,x));\\
Q:=List([1..Size(pi)],i->Size(N[i])/Size(C[i]));\\
R:=List([1..Size(pi)],i->ForAny(FactorsInt(Q[i]),p->p<pi[i]));\\
\medskip
if pi[1]=2 then\\
res:=ForAll(R\{[2..Size(pi)]\},x->x);\\
else res:=ForAll(R,x->x);\\
fi;\\
return
res;\\
 end;}

\section*{acknowledgement}  The present work was supported by GNSAGA (Indam, Florence, Italy) in the years
2008--09. This allowed me to visit Michigan State University (USA), TU of Darmstadt (Germany), Free University of
Berlin (Germany) and University Technology of Malaysia (Malaysia), finding the inspiration for the present paper.


\section{Appendix}

\scriptsize{

\begin{itemize}
\item[(1)] For $G=M_{11}$, $\pi(M_{11})=\{2,3,5,11\}$.
$2\in \textrm{Div} (|N_{M_{11}}(G_3):C_{M_{11}}(G_3)|)\cap
\textrm{Div} (|N_G(G_5):C_G(G_5)|)\cap \textrm{Div}
(|N_G(G_{11}):C_G(G_{11})|),$ $ 5\in \textrm{Div}
(|N_G(G_{11}):C_G(G_{11})|)$. These data are enough.
\item[(2)] For $G=M_{12}$, $\pi(M_{12})=\{2,3,5,11\}$. $2\in \textrm{Div} (|N_G(G_3):C_G(G_3)|)\cap \textrm{Div} (|N_G(G_5):C_G(G_5)|)\cap \textrm{Div} (|N_G(G_{11}):C_G(G_{11})|),$ $5\in \textrm{Div} (|N_G(G_{11}):C_G(G_{11})|),$
These data are enough.
\item[(3)] For $G=J_1$, $\pi(J_1)=\{2,3,5,7,11, 19\}$. $2\in \textrm{Div} (|N_G(G_3):C_G(G_3)|)\cap \textrm{Div} (|N_G(G_5):C_G(G_5)|),$ $ 3\in \textrm{Div} (|N_G(G_7):C_G(G_7)|)\cap \textrm{Div} (|N_G(G_{19}):C_G(G_{19})|),
$ $5\in \textrm{Div} (|N_G(G_{11}):C_G(G_{11})|)$. These data are
enough.
\item[(4)] For $G=M_{22}$, $\pi(M_{22})=\{2,3,5,7,11\}$. $2\in \textrm{Div} (|N_G(G_3):C_G(G_3)|)\cap \textrm{Div} (|N_G(G_5):C_G(G_5)|),$ $3\in \textrm{Div} (|N_G(G_7):C_G(G_7)|),$ $5\in \textrm{Div} (|N_G(G_{11}):C_G(G_{11})|).$
These data are enough.
\item[(5)]  For $G=J_2$,  $\pi(J_2)=\{2,3,5,7\}$. $2\in \textrm{Div} (|N_G(G_3):C_G(G_3)|)\cap \textrm{Div} (|N_G(G_5):C_G(G_5)|), 3\in \textrm{Div} (|N_G(G_7):C_G(G_7)|).$
These data are enough.
\item[(6)] For $G=M_{23}$,  $\pi(M_{23})=\{2,3,5,7,11,23\}$. $2\in \textrm{Div} (|N_G(G_3):C_G(G_3)|)\cap \textrm{Div} (|N_G(G_5):C_G(G_5)|),$ $3\in \textrm{Div} (|N_G(G_7):C_G(G_7)|),$
$5\in \textrm{Div} (|N_G(G_{11}):C_G(G_{11})|),$ $11 \in
\textrm{Div} (|N_G(G_{23}):C_G(G_{23})|)$. These data are enough.
\item[(7)] For $G=H_5$,  $\pi(H_5)=\{2,3,5,7,11\}$. $2\in \textrm{Div} (|N_G(G_3):C_G(G_3)|)\cap \textrm{Div} (|N_G(G_5):C_G(G_5)|),$
$3\in \textrm{Div} (|N_G(G_7):C_G(G_7)|),$ $5\in \textrm{Div}
(|N_G(G_{11}):C_G(G_{11})|).$ These data are enough.
\item[(8)]For $G=J_3$,  $\pi(J_3)=\{2,3,5,17,19\}$. $2\in \textrm{Div} (|N_G(G_3):C_G(G_3)|)\cap \textrm{Div} (|N_G(G_5):C_G(G_5)|) \cap \textrm{Div} (|N_G(G_7):C_G(G_7)|),$
$3\in \textrm{Div} (|N_G(G_{19}):C_G(G_{19})|).$ These data are
enough.
\item[(9)] For $G=M_{24}$, $\pi(M_{24})=\{2,3,5,17,19\}$. $2\in \textrm{Div} (|N_G(G_3):C_G(G_3)|)\cap \textrm{Div} (|N_G(G_5):C_G(G_5)|) \cap \textrm{Div} (|N_G(G_7):C_G(G_7)|),$
$3\in \textrm{Div} (|N_G(G_{19}):C_G(G_{19})|).$ These data are
enough.
\item[(10)] For $G=McL$, $\pi(McL)=\{2,3,5,7,11\}$.$2\in \textrm{Div} (|N_G(G_3):C_G(G_3)|)\cap \textrm{Div} (|N_G(G_5):C_G(G_5)|),$ $3\in \textrm{Div} (|N_G(G_{7}):C_G(G_{7})|),$ $5\in \textrm{Div}
(|N_G(G_{11}):C_G(G_{11})|).$ These data are enough.
\item[(11)] For $G=He$, $\pi(He)=\{2,3,5,7,17\}$. $2\in \textrm{Div} (|N_G(G_3):C_G(G_3)|)\cap \textrm{Div} (|N_G(G_5):C_G(G_5)|)\cap \textrm{Div} (|N_G(G_{17}):C_G(G_{17})|),$
$3\in \textrm{Div} (|N_G(G_{7}):C_G(G_{7})|).$ These data are
enough.
\item[(12)]For $G=Ru$,  $\pi(Ru)=\{2,3,5,7,13,29\}$. $2\in \textrm{Div} (|N_G(G_3):C_G(G_3)|)\cap \textrm{Div} (|N_G(G_5):C_G(G_5)|),$
$3\in \textrm{Div} (|N_G(G_{7}):C_G(G_{7})|)\cap \textrm{Div}
(|N_G(G_{13}):C_G(G_{13})|),$ $7\in \textrm{Div}
(|N_G(G_{29}):C_G(G_{29})|).$ These data are enough.
\item[(13)]For $G=Suz$, $\pi(Suz)=\{2,3,5,7,11,13\}$. $2\in \textrm{Div} (|N_G(G_3):C_G(G_3)|)\cap \textrm{Div} (|N_G(G_5):C_G(G_5)|),$
$3\in \textrm{Div} (|N_G(G_{7}):C_G(G_{7})|)\cap \textrm{Div}
(|N_G(G_{13}):C_G(G_{13})|),$ $5\in \textrm{Div}
(|N_G(G_{11}):C_G(G_{11})|).$ These data are enough.
\item[(14)] For $G=O'N$, $\pi(O'N)=\{2,3,5,7,11,19,31\}$. $2\in \textrm{Div} (|N_G(G_3):C_G(G_3)|)\cap \textrm{Div} (|N_G(G_5):C_G(G_5)|),$
$3\in \textrm{Div} (|N_G(G_{7}):C_G(G_{7})|)\cap \textrm{Div}
(|N_G(G_{19}):C_G(G_{19})|),$ $5\in \textrm{Div}
(|N_G(G_{11}):C_G(G_{11})|)\cap \textrm{Div}
(|N_G(G_{31}):C_G(G_{31})|).$ These data are enough.
\item[(15)]For $G=Co3$, $\pi(Co3)=\{2,3,5,7,11,23\}$.
$2\in \textrm{Div} (|N_G(G_3):C_G(G_3)|)\cap \textrm{Div}
(|N_G(G_5):C_G(G_5)|),$ $3\in \textrm{Div}
(|N_G(G_{7}):C_G(G_{7})|),$ $5\in \textrm{Div}
(|N_G(G_{11}):C_G(G_{11})|),$ $11\in \textrm{Div}
(|N_G(G_{23}):C_G(G_{23})|).$ These data are enough.
\item[(16)]For $G=Co2$, $\pi(Co2)=\{2,3,5,7,11,23\}$.
$2\in \textrm{Div} (|N_G(G_3):C_G(G_3)|)\cap \textrm{Div}
(|N_G(G_5):C_G(G_5)|),$ $3\in \textrm{Div}
(|N_G(G_{7}):C_G(G_{7})|),$ $5\in \textrm{Div}
(|N_G(G_{11}):C_G(G_{11})|),$ $11\in \textrm{Div}
(|N_G(G_{23}):C_G(G_{23})|).$ These data are enough.
\item[(17)]For $G=Co1$, $\pi(Co1)=\{2,3,5,7,11,13,23\}$.
$2\in \textrm{Div} (|N_G(G_3):C_G(G_3)|)\cap \textrm{Div}
(|N_G(G_5):C_G(G_5)|),$ $3\in \textrm{Div}
(|N_G(G_{7}):C_G(G_{7})|)\cap \textrm{Div}
(|N_G(G_{13}):C_G(G_{13})|),$ $5\in \textrm{Div}
(|N_G(G_{11}):C_G(G_{11})|),$ $11\in \textrm{Div}
(|N_G(G_{23}):C_G(G_{23})|).$ These data are enough.
\item[(18)]For $G=Fi_{22}$, $\pi(Fi_{22})=\{2,3,5,7,11,13\}$.
$2\in \textrm{Div} (|N_G(G_3):C_G(G_3)|)\cap \textrm{Div}
(|N_G(G_5):C_G(G_5)|),$ $3\in \textrm{Div}
(|N_G(G_{7}):C_G(G_{7})|)\cap \textrm{Div}
(|N_G(G_{13}):C_G(G_{13})|),$ $5\in \textrm{Div}
(|N_G(G_{11}):C_G(G_{11})|).$ These data are enough.
\item[(19)] For$G=Fi_{24}$, $\pi(Fi_{24})=\{2,3,5,7,11,13,17,23,29\}$. $2\in \textrm{Div} (|N_G(G_3):C_G(G_3)|)\cap \textrm{Div} (|N_G(G_5):C_G(G_5)|)\cap \textrm{Div} (|N_G(G_{17}):C_G(G_{17})|),$
$3\in \textrm{Div} (|N_G(G_{7}):C_G(G_{7})|)\cap \textrm{Div}
(|N_G(G_{13}):C_G(G_{13})|),$ $5\in \textrm{Div}
(|N_G(G_{11}):C_G(G_{11})|),$ $7\in \textrm{Div}
(|N_G(G_{29}):C_G(G_{29})|),$ $11\in \textrm{Div}
(|N_G(G_{23}):C_G(G_{23})|),$ These data are enough.
\item[(20)] For $G=HN$, $\pi(HN)=\{2,3,5,7,11,19\}$. $2\in \textrm{Div} (|N_G(G_3):C_G(G_3)|)\cap \textrm{Div} (|N_G(G_5):C_G(G_5)|),$ $3\in \textrm{Div} (|N_G(G_{7}):C_G(G_{7})|)\cap \textrm{Div} (|N_G(G_{19}):C_G(G_{19})|),$
$5\in \textrm{Div} (|N_G(G_{11}):C_G(G_{11})|).$ These data are
enough.
\item[(21)] For $G=Ly$,  $\pi(Ly)=\{2,3,5,7,11,31,37,67\}$. $2\in \textrm{Div} (|N_G(G_3):C_G(G_3)|)\cap \textrm{Div} (|N_G(G_5):C_G(G_5)|),$
$3\in \textrm{Div} (|N_G(G_{7}):C_G(G_{7})|)\cap \textrm{Div}
(|N_G(G_{37}):C_G(G_{37})|),$ $5\in \textrm{Div}
(|N_G(G_{11}):C_G(G_{11})|)\cap \textrm{Div}
(|N_G(G_{31}):C_G(G_{31})|)l,$ $11\in \textrm{Div}
(|N_G(G_{37}):C_G(G_{37})|).$ These data are enough.
\item[(22)] For $G=Fi_{23}$, $\pi(Fi_{24})=\{2,3,5,7,11,13,17,23\}$.
$2\in \textrm{Div} (|N_G(G_3):C_G(G_3)|)\cap \textrm{Div}
(|N_G(G_5):C_G(G_5)|)\cap \textrm{Div} (|N_G(G_{17}):C_G(G_{17})|),$
$3\in \textrm{Div} (|N_G(G_{7}):C_G(G_{7})|)\cap \textrm{Div}
(|N_G(G_{13}):C_G(G_{13})|),$ $5\in \textrm{Div}
(|N_G(G_{11}):C_G(G_{11})|),$ $11\in \textrm{Div}
(|N_G(G_{23}):C_G(G_{23})|).$ These data are enough.
\item[(23)]For $G=Th$, $\pi(Th)=\{2,3,5,7,11,13,19,31\}$.
$2\in \textrm{Div} (|N_G(G_3):C_G(G_3)|)\cap \textrm{Div}
(|N_G(G_5):C_G(G_5)|),$ $3\in \textrm{Div}
(|N_G(G_{7}):C_G(G_{7})|)\cap \textrm{Div}
(|N_G(G_{13}):C_G(G_{13})|)\cap \textrm{Div}
(|N_G(G_{19}):C_G(G_{19})|),$ $5\in \textrm{Div}
(|N_G(G_{31}):C_G(G_{31})|).$ These data are enough.
\item[(24)]For $G=J_4$,  $\pi(J_4)=\{2,3,5,7,11,23,29,31,37,43\}$.
$2\in \textrm{Div} (|N_G(G_3):C_G(G_3)|)\cap \textrm{Div}
(|N_G(G_5):C_G(G_5)|),$ $3\in \textrm{Div}
(|N_G(G_{7}):C_G(G_{7})|)\cap \textrm{Div}
(|N_G(G_{23}):C_G(G_{23})|) \cap \textrm{Div}
(|N_G(G_{37}):C_G(G_{37})|),$ $5\in \textrm{Div}
(|N_G(G_{11}):C_G(G_{11})|)\cap \textrm{Div}
(|N_G(G_{31}):C_G(G_{31})|),$ $7\in \textrm{Div}
(|N_G(G_{29}):C_G(G_{29})|)\cap \textrm{Div}
(|N_G(G_{43}):C_G(G_{43})|).$ These data are enough.
\item[(25)]For $G=B$,  $\pi(B)=\{2,3,5,7,11,13,17,19,23,31,47\}$.
$2\in \textrm{Div} (|N_G(G_3):C_G(G_3)|)\cap \textrm{Div}
(|N_G(G_5):C_G(G_5)|)\cap \textrm{Div} (|N_G(G_{17}):C_G(G_{17})|),$
$3\in \textrm{Div} (|N_G(G_{7}):C_G(G_{7})|)\cap \textrm{Div}
(|N_G(G_{13}):C_G(G_{13})|)$$ \cap \textrm{Div}
(|N_G(G_{19}):C_G(G_{19})|)\cap \textrm{Div}
(|N_G(G_{37}):C_G(G_{37})|),$ $5\in \textrm{Div}
(|N_G(G_{11}):C_G(G_{11})|)\cap \textrm{Div}
(|N_G(G_{23}):C_G(G_{23})|)\cap \textrm{Div}
(|N_G(G_{31}):C_G(G_{31})|),$ $23\in \textrm{Div}
(|N_G(G_{47}):C_G(G_{47})|).$ These data are enough.
\item[(26)]For  $G=M$, $\pi(M)=\{2,3,5,7,11,13,17,19,23,29,31,41,47,59,71\}$.
$2\in \textrm{Div} (|N_G(G_3):C_G(G_3)|)\cap \textrm{Div}
(|N_G(G_5):C_G(G_5)|)\cap \textrm{Div} (|N_G(G_{17}):C_G(G_{17})|),$
$3\in \textrm{Div} (|N_G(G_{7}):C_G(G_{7})|)\cap \textrm{Div}
(|N_G(G_{13}):C_G(G_{13})|) $$\cap \textrm{Div}
(|N_G(G_{19}):C_G(G_{19})|)\cap \textrm{Div}
(|N_G(G_{37}):C_G(G_{37})|),$ \newline 
$5\in \textrm{Div}
(|N_G(G_{11}):C_G(G_{11})|)\cap  \textrm{Div}
(|N_G(G_{23}):C_G(G_{23})|)$$\cap \textrm{Div}
(|N_G(G_{31}):C_G(G_{31})|)\cap \textrm{Div}
(|N_G(G_{41}):C_G(G_{41})|),$$7\in \textrm{Div}
(|N_G(G_{71}):C_G(G_{71})|),$ $23\in \textrm{Div}
(|N_G(G_{47}):C_G(G_{47})|),$  \newline
$29\in \textrm{Div}
(|N_G(G_{59}):C_G(G_{59})|).$ These data are enough.
\end{itemize}

}

\end{document}